\documentclass[12pt, reqno]{amsart}
\usepackage{amsmath, amsthm, amscd, amsfonts, amssymb, graphicx, color}
\usepackage[bookmarksnumbered, colorlinks, plainpages]{hyperref}

\input{mathrsfs.sty}

\newtheorem{theorem}{Theorem}[section]
\newtheorem{lemma}[theorem]{Lemma}
\newtheorem{proposition}[theorem]{Proposition}
\newtheorem{corollary}[theorem]{Corollary}
\theoremstyle{definition}

\newtheorem{example}[theorem]{Example}
\theoremstyle{remark}
\newtheorem{remark}[theorem]{Remark}

\begin{document}
\title[Fuglede--Putnam type theorems]{Fuglede--Putnam type theorems via the Aluthge transform}
\author[M.S. Moslehian, S.M.S. Nabavi Sales]{M. S. Moslehian$^1$ and S. M. S. Nabavi Sales$^2$}
\address{$^1$ Department of Pure Mathematics, Center of Excellence in
Analysis on Algebraic Structures (CEAAS), Ferdowsi University of
Mashhad, P. O. Box 1159, Mashhad 91775, Iran.}
\email{moslehian@ferdowsi.um.ac.ir}
\urladdr{\url{http://profsite.um.ac.ir/~moslehian/}}
\address{$^{2}$ Department of Pure Mathematics, Ferdowsi University of
Mashhad, P. O. Box 1159, Mashhad 91775, Iran.}
\email{sadegh.nabavi@gmail.com}



\begin{abstract}
Let $A=U|A|$ and $B=V|B|$ be the polar decompositions of $A\in
\mathbb{B}(\mathscr{H}_1)$ and $B\in \mathbb{B}(\mathscr{H}_2)$ and
let $\mbox{Com}(A,B)$ stand for the set of operators
$X\in\mathbb{B}(\mathscr{H}_2,\mathscr{H}_1)$ such that $AX=XB$. A
pair $(A,B)$ is said to have the FP-property if
$\mbox{Com}(A,B)\subseteq\mbox{Com}(A^\ast,B^\ast)$. Let $\tilde{C}$
denote the Aluthge transform of a bounded linear operator $C$. We
show that (i) if $A$ and $B$ are invertible and $(A,B)$ has the
FP-property, then so is $(\tilde{A},\tilde{B})$; (ii) if $A$ and $B$
are invertible, the spectrums of both $U$ and $V$ are contained in
some open semicircle and $(\tilde{A},\tilde{B})$ has the
FP-property, then so is $(A,B)$; (iii) if $(A,B)$ has the
FP-property, then
$\mbox{Com}(A,B)\subseteq\mbox{Com}(\tilde{A},\tilde{B})$, moreover,
if $A$ is invertible, then
$\mbox{Com}(A,B)=\mbox{Com}(\tilde{A},\tilde{B})$. Finally, if
$\mbox{Re}(U|A|^{1\over2})\geq a>0$ and
$\mbox{Re}(V|B|^{1\over2})\geq a>0$ and $X$ is an operator such that
$U^* X=XV$, then we prove that $\|\tilde{A}^* X-X\tilde{B}\|_p\geq
2a\|\,|B|^{1\over2}X-X|B|^{1\over2}\|_p$ for any $1 \leq p \leq
\infty$.
\end{abstract}
\maketitle
\section{Introduction and preliminaries}

Let $\mathbb{B}(\mathscr{H}_1, \mathscr{H}_2)$ be the algebra of all bounded linear
operators between (separable) complex Hilbert spaces $\mathscr{H}_1$ and $\mathscr{H}_2$, let $\mathbb{B}(\mathscr{H})$
denote $\mathbb{B}(\mathscr{H}, \mathscr{H})$ and let $I \in \mathbb{B}(\mathscr{H})$ be the identity operator. A subspace
$\mathscr{K}\subseteq \mathscr{H}$ is said to reduce $A\in\mathbb{B}(\mathscr{H})$ if
$A\mathscr{K}\subseteq \mathscr{K}$ and $A^* \mathscr{K}\subseteq \mathscr{K}$. Let $\mathbb{K}(\mathscr{H})$ denote the two-sided ideal of all compact operators on $\mathscr{H}$. For any
compact operator $A$, let $s_1(A), s_2(A), \ldots$ be the singular
values of $A$, i.e., the eigenvalues of $|A|=(A^* A)^{1\over2}$ in
decreasing order and repeated according to the multiplicity. If
$\sum_{i=1}^\infty s_i(A)^p<\infty$, for some $1\leq p<\infty$, we
say that $A$ is in the Schatten class $\mathcal{C}_p$ and
$\|A\|_p=(\sum_{i=1}^\infty s_i(A)^p)^{1\over p}$ is called the
Schatten $p$-norm of $A$. This norm makes $\mathcal{C}_p$ into a Banach space.
Note that $\mathcal{C}_1$ is the trace class and $\mathcal{C}_2$ is the Hilbert-Schmidt
class. It is convenient to put $\mathcal{C}_\infty=\mathbb{K}(\mathscr{H})$ and to denote the usual
operator norm $\|.\|$ by $\|.\|_\infty$. If $\{e_i\}_{i=1}^\infty$ and
$\{f_i\}_{i=1}^\infty$ are two orthonormal families in $\mathscr{H}$, then for
$A\in \mathcal{C}_p$, $\|A\|_p^p\geq\sum_{i=1}^\infty|\langle
Ae_i,f_i\rangle|^p$. If $A,B\in \mathcal{C}_p$, then
\begin{eqnarray}\label{1new}
\left\|\left(
  \begin{array}{cc}
  0 & A \\
  B & 0 \\
  \end{array}
 \right)\right\|_p^p&=&\|A\|_p^p+\|B\|_p^p
\qquad\quad (0\leq p<\infty)\,,
\nonumber\\
&&\\
\left\|\left(
  \begin{array}{cc}
  0 & A \\
  B & 0 \\
  \end{array}
 \right)\right\|_{\infty}&=&\max (\|A\|_{\infty},\|B\|_{\infty})\,.\nonumber
 \end{eqnarray}
We refer the reader to \cite{ALU2} for further
properties of the Schatten $p$-classes.

For $p>0$, an operator $A$ is called $p$-hyponormal if $(A^*
A)^p\geq (AA^*)^p$. If $A$ is an invertible operator satisfying
$\log(A^* A)\geq \log(AA^*)$, then it is called log-hyponormal. If
$p=1$, then $A$ is said to be hyponormal. If $A$ is invertible and
$p$-hyponormal then $A$ is log-hyponormal.

Let $A=U|A|$ be the polar decomposition of $A$. It is known that if
$A$ is invertible then $U$ is unitary and $|A|$ is also invertible.
The Aluthge transform $\tilde{A}$ of $A$ is defined by
$\tilde{A}:=|A|^{1\over2}U|A|^{1\over2}$. This notion was first
introduced by Aluthge \cite{ALU} and is a powerful tool in the
operator theory. There are some significant evidences for this
assertion, for instance, it is proved in \cite{JKP} that any
operator $A$ has a nontrivial invariant subspace if and only if so
does $\tilde{A}$. Another interesting application deals with an
application of the Aluthge transform for generalizing the
Fuglede--Putnam theorem \cite{JTU}. It indeed is a motivation for
our work in this paper. Let $A\in\mathbb{B}(\mathscr{H}_1)$ and
$B\in\mathbb{B}(\mathscr{H}_2)$. For such pair $(A,B)$, denote by
$\mbox{Com}(A,B)$ the set of operators
$X\in\mathbb{B}(\mathscr{H}_2,\mathscr{H}_1)$ such that $AX=XB$. A
pair $(A,B)$ is said to have the FP-property if
$\mbox{Com}(A,B)\subseteq\mbox{Com}(A^\ast,B^\ast)$. The
Fuglede--Putnam theorem is well-known in the operator theory. It
asserts that for any normal operators $A$ and $B$, the pair $(A,B)$
has the FP-property. First Fuglede \cite{FUG} proved it in the case
when $A=B$ and then Putnam \cite{PUT} proved it in a general case.
There exist many generalizations of this theorem which most of them
go into relaxing the normality of $A$ and $B$; see \cite{FUR2, RAD,
JOC, DUG1, DUG2, MTU, MOR} and references therein. The two next
lemmas are concerned with the Fuglede--Putnam theorem and we need
them in the future.
\begin{lemma}\label{L1}\cite{TAK}
Let $A\in\mathbb{B}(\mathscr{H}_1)$ and
$B\in\mathbb{B}(\mathscr{H}_2)$. Then the following assertions are
equivalent

{\rm(i)} The pair $(A,B)$ has the FP-property.

{\rm (ii)} If $X\in\mbox{Com}(A,B)$, then $\overline{R(X)}$ reduces
$A$, $(\ker X)^\bot$ reduces $B$, and $A|_{\overline{R(X)}}$, $B|_{(\ker
X)^\bot}$ are unitarily equivalent normal operators.
\end{lemma}
\begin{lemma}\label{L3}\cite{JTU}
Let $A\in \mathbb{B}(\mathscr{H}_1)$ and $B^*\in
\mathbb{B}(\mathscr{H}_2)$ be either log-hyponormal or
$p$-hyponormal
 operators. Then the pair $(A,B)$ has the FP-property.
\end{lemma}

Recently some investigation in the operator theory have been related
to relationship between operators and their Aluthge transform; see
\cite{AMS, GAR, IYY, MOS3, TPA, YAM}. In this paper we present some
results that are in the same direction but of some new views of
points via the Fuglede--Putnam theorem. For instance, one of our
problems is as follows: Under what conditions on operators $A$ and
$B$ does the FP-property for the pair $(A,B)$ imply that for
$(\tilde{A},\tilde{B})$? Another question is related to the
converse. In Section 2 we try to answer these questions.

The iterated Aluthge transforms of $A$ are the operators
$\Delta_n(A)$ defined by $\Delta_1(A):=\tilde{A}$ and
$\Delta_n(A):=\Delta_1(\Delta_{n-1}(A))$ for $n>1$. A surprising
fact about these operators is the convergence of their norms to the
spectral radius of $A$; cf. \cite{YAM}. Also the convergence of the
sequence of iterates is an interesting question, which is recently
investigated in \cite{APS}. In section 2 and 3, we provide some
results about these operators as well.

Another interesting problem is that under what conditions on $A, B,
X$, any one of $AX=XB$ and $\tilde{A}X=X\tilde{B}$ implies the
other. In Section 3 we try to provide some results concerning this
problem that we call it the Fuglede--Putnam--Aluthge problem. More
precisely, we prove that if $(A,B)$ has the FP-property, then
$\mbox{Com}(A,B)\subseteq\mbox{Com}(\tilde{A},\tilde{B})$ and if,
moreover, $A$ is invertible operator then
$\mbox{Com}(A,B)=\mbox{Com}(\tilde{A},\tilde{B})$. We also study
Fuglede--Putnam--Aluthge problem modulo trace ideals and give
several Schatten $p$-norm inequalities in Section 4; see also
\cite{MOS1, MOS2}. The reader is referred to \cite{FUR} for
undefined notions and terminology.

\section{Fuglede--Putnam theorem for the Aluthge transforms}
In this section we assume that $A\in\mathbb{B}(\mathscr{H}_1)$ and
$B\in\mathbb{B}(\mathscr{H}_2)$ are invertible operators with the
polar decompositions $A=U|A|$ and $B=V|B|$, where $U$ and $V$ are
unitaries.
\begin{lemma}\label{A1}$\:$

 {\rm(i)}
$X\in\mbox{Com}(A,B)\:\:\:\Longleftrightarrow\:\:\:
|A|X|B|^{-1}=U^\ast XV$\,;

 {\rm(ii)}
$X\in\mbox{Com}(A,B)\bigcap\mbox{Com}(A^\ast,B^\ast)\:\:\:\Longleftrightarrow\:\:\:|A|X|B|^{-1}=U^\ast
XV=X$\,.
\end{lemma}
\begin{proof}
{\rm(i)} is just the definition itself.

{\rm(ii)} Let $X\in\mbox{Com}(A,B)\bigcap\mbox{Com}(A^\ast,B^\ast)$.
Then $|A|^2X=X|B|^2$. Utilizing a sequence of polynomials uniformly
converging to $f(t)=\sqrt{t}$ on ${\rm sp}(|A|^2)\cup {\rm
sp}(|B|^2)$ and the functional calculus
 we get
$|A|X=X|B|$, that is $|A|X|B|^{-1}=X$. Hence from {\rm(i)} we have
$U^\ast XV=X$. The reverse direction is trivial.
\end{proof}
\begin{remark}\label{A2} The proof of Lemma \ref{A1} shows that if
$X\in\mbox{Com}(A,B)\bigcap\mbox{Com}(A^\ast,B^\ast)$ and $p$ be a
positive number, then $|A|^pX=X|B|^p$.
\end{remark}
\begin{lemma}\label{A3}$\:$

{\rm(i)}
$X\in\mbox{Com}(A,B)\:\:\:\Longleftrightarrow\:\:\:|A|^{1\over2}X|B|^{-1\over2}\in\mbox{Com}(\tilde{A},\tilde{B})$\,;

{\rm(ii)}$X\in\mbox{Com}(A^\ast,B^\ast)\:\:\:\Longleftrightarrow\:\:\:|A|^{-1\over2}X|B|^{1\over2}\in\mbox{Com}((\tilde{A})^\ast,(\tilde{B})^\ast)$\,.
\end{lemma}
\begin{proof}
{\rm(i)}Let $AX=XB$ for some $X\in \mathbb{B}(\mathscr{H})$. Then
\begin{eqnarray*}\label{5}
U|A|X&=&XV|B|.
\end{eqnarray*}
Hence
\begin{eqnarray*}\label{6}
\tilde{A}(|A|^{1\over2}X|B|^{-1\over2})&=&
|A|^{1\over2}(U|A|^{1\over2}|A|^{1\over2}X)|B|^{-1\over2}\nonumber\\&=&
|A|^{1\over2}(X|B|^{-1\over2}|B|^{1\over2}V|B|)|B|^{-1\over2}\nonumber\\
&=&(|A|^{1\over2}X|B|^{-1\over2})\tilde{B}.
\end{eqnarray*}
The converse obviously holds.

{\rm(ii)} It can be proved in a similar way to {\rm(i)}.
\end{proof}
\begin{theorem}\label{A4}
The pair $(\tilde{A},\tilde{B})$ has the FP-property, i.e.
$\mbox{Com}(\tilde{A},\tilde{B})\subseteq\mbox{Com}((\tilde{A})^\ast,(\tilde{B})^\ast)$
if and only if $U^2X=XV^2$ for any $X\in\mbox{Com}(A,B)$.
\end{theorem}
\begin{proof}First we show that the FP-property for
$(\tilde{A},\tilde{B})$ is equivalent to the following requirement
\begin{eqnarray}\label{B1}
|A|X|B|^{-1}\in\mbox{Com}(A^\ast,B^\ast)\:\:\:\:\:(X\in\mbox{Com}(A,B)).
\end{eqnarray}
Let$(\tilde{A},\tilde{B})$ have the FP-property and
$X\in\mbox{Com}(A,B)$. By Lemma \ref{A3}{\rm(i)},
$|A|^{1\over2}X|B|^{-1\over2}\in\mbox{Com}(\tilde{A},\tilde{B}).$
Since $(\tilde{A},\tilde{B})$ has the FP-property we have
$|A|^{1\over2}X|B|^{-1\over2}\in\mbox{Com}((\tilde{A})^\ast,(\tilde{B})^\ast).$
By Lemma \ref{A3}{\rm(ii)} we have
$|A|X|B|^{-1}\in\mbox{Com}(A^\ast,B^\ast)$, so we reach \eqref{B1}.
To prove the revers, assume the assertion \eqref{B1} and let
$X\in\mbox{Com}(\tilde{A},\tilde{B})$. It follows from Lemma
\ref{A3}{\rm(i)} that
$|A|^{-1\over2}X|B|^{1\over2}\in\mbox{Com}(A,B)$. Hence by
\eqref{B1} we have
$|A|^{1\over2}X|B|^{-1\over2}\in\mbox{Com}(A^\ast,B^\ast)$ which in
turn implies that
$X\in\mbox{Com}((\tilde{A})^\ast,(\tilde{B})^\ast)$. Thus
$(\tilde{A},\tilde{B})$ has the FP-property.

Let \eqref{B1} hold. For any $X\in\mbox{Com}(A,B)$ it follows from
Lemma \ref{A1}{\rm(i)} that $|A|X|B|^{-1}=U^\ast XV$. Using
\eqref{B1} we obtain
$$|A|U^\ast U^\ast
XV=|A|X|B|^{-1}|B|V^\ast\:\:\:\:(X\in\mbox{Com}(A,B)),$$ which
simply becomes $U^2X=XV^2$ for any $X\in\mbox{Com}(A,B)$. The
converse can be proved in a similar fashion.
\end{proof}
\begin{corollary}\label{DD}
If $(A,B)$ has the FP-property, then so is $(\tilde{A},\tilde{B})$.
\end{corollary}
\begin{proof}
If $(A,B)$ has the FP-property, then by Lemma \ref{A1}{\rm(ii)}
$UX=XV$ for any $X\in\mbox{Com}(A,B)$. Hence $U^2X=XV^2$. Applying
Theorem \ref{A4} we observe that $(\tilde{A},\tilde{B})$ has the
FP-property.
\end{proof}
\begin{corollary}\label{A6}
If $(A,B)$ has the FP-property, then so is
$(\Delta_n(A),\Delta_n(B))$ for any positive integer $n$.
\end{corollary}
\begin{corollary}\label{B2}
If the spectrums of both $U$ and $V$ are contained in some open
semicircle, then the FP-property for $(A,B)$ is equivalent to the
FP-property for $(\tilde{A},\tilde{B})$.
\end{corollary}
\begin{proof}
Let $(\tilde{A},\tilde{B})$ have the FP-property and
$X\in\mbox{Com}(A,B)$. Then $U^2X=XV^2$ by Theorem \ref{A4}. Under
the spectral conditions on $U$ and $V$ the unitary operator $U$
(resp. $V$) can be approximated by polynomials of $U^2$ (resp.
$V^2$), therefore $U^2X=XV^2$ implies $UX=XV$, that is, $U^*XV=X$
and this by Lemma \ref{A1} implies that
$X\in\mbox{Com}(A^\ast,B^\ast)$. The rest follows from Corollary
\ref{DD}.
\end{proof}
\begin{remark}
Note that if the conditions on $U$ and $V$ in Corollary \ref{B2}
are replaced by the condition that $U^{2n_0+1}=V^{2n_0+1}=I$ for
some positive integer $n_0$, then we obtain the same result. In fact
from $U^2X=XV^2$ we get $U^{2n_0}X=XV^{2n_0}$ that under our
assumption implies that $UX=XV$.
\end{remark}\label{A5}
An interesting problem is that under what conditions on operator
$A$, $A^n=I$ implies that $U^n=I$, where $A=U|A|$ is the polar
decomposition of $A$ and $n\geq1$. It is known that for a normaloid
operator $A$, $A^n=I$ implies that $A$ is unitary \cite[Corollary
3.7.3.6]{FUR}. The next result is related to this problem.

\begin{proposition}\label{B3}
Let $A=U|A|$ be the polar decomposition of $A$ and $A^2=I$ then
$U^2=I$.
\end{proposition}
\begin{proof}
Since $A^2=I$ we have
\begin{eqnarray}\label{15}
U|A|U|A|&=&I.
\end{eqnarray}
We multiply both side of \eqref{15} by $|A|^{-1}$ to obtain
\begin{eqnarray*}\label{16}
|A|^{-1}&=&U|A|U=U^2U^\ast|A|U.
\end{eqnarray*}
Since $U^2$ is unitary and $U^\ast|A|U\geq0$ and view of uniqueness
of polar decomposition of $|A|^{-1}$, the unitary operator $U^2$
should coincide with the angular part $I$ of positive definite
$|A|^{-1}$.
\end{proof}
\begin{remark}
The above proposition is a consequence of \cite[Theorem 2.1]{IYY2},
which states that if $T=U|T|, S=V|S|$ and $|T||S^*|=W|\,|T||S^*|\,|$
are the polar decompositions, then $TS=UWV|TS|$ is also the polar
decomposition.
\end{remark}
\begin{example} Proposition \ref{B3} is not valid when the power $2$ is replaced by $3$. Indeed there exists an operator
$A$ with the polar decomposition $A=U|A|$ such that $A^3=I$ but
$U^3\neq I$. To see this let $A=\left(
\begin{array}{cc}
   0 & 1 \\
   -1 & -1 \\
  \end{array}
 \right)$. Then $|A|=
 \left(
  \begin{array}{cc}
   {2\sqrt{5}}\over5 & {\sqrt{5}}\over5 \\
   {\sqrt{5}}\over5 & {3\sqrt{5}}\over5 \\
  \end{array}
 \right)$ and $U=A|A|^{-1}=
 \left(
  \begin{array}{cc}
   {-\sqrt{5}}\over5 & {2\sqrt{5}}\over5 \\
   {-2\sqrt{5}}\over5 & {-\sqrt{5}}\over5 \\
  \end{array}
 \right)$. It is easy to verify that $A^3=I$ and $U^3=
 \left(
  \begin{array}{cc}
   {11\sqrt{5}}\over25 & {-2\sqrt{5}}\over25 \\
   {-2\sqrt{5}}\over25 & {3\sqrt{5}}\over25 \\
  \end{array}
 \right)\neq I$\,.
\end{example}
Now we present an example to show that in Corollary \ref{B2} and
Remark \ref{A5} the conditions are essential.
\begin{example}
Let $A=\left(
  \begin{array}{cc}
   2 & -3 \\
   1 & -2 \\
  \end{array}
 \right)$ and $X=\left(
  \begin{array}{cc}
   0 & -3 \\
   1 & -4 \\
  \end{array}
 \right)$. It is easy to verify that
$AX=XA$ and $A^\ast X\neq XA^\ast$. On the other hand, an easy
computation shows that $A^2=I$. Hence by Proposition \ref{B3},
$U^2=I$ in which $A=U|A|$ is the polar decomposition of $A$. Hence
$U=U^\ast$, so that $\tilde{A}=|A|^{1\over2}U|A|^{1\over2}$ is self
adjoint. Thus $\tilde{A}X=X\tilde{A}$ implies $\tilde{A}^\ast
X=X\tilde{A}^\ast$ for any $X$.
\end{example}

\section{The Fuglede--Putnam--Aluthge problem}

In this section, we present some results concerning the
Fuglede--Putnam--Aluthge problem without assumption of invertibility
of $A$ and $B$, in general.

\begin{theorem}\label{C1}
Let $A\in \mathbb{B}(\mathscr{H}_1)$, $B\in
\mathbb{B}(\mathscr{H}_2)$ and $(A,B)$ have the FP-property. Then
$\mbox{Com}(A,B)\subseteq\mbox{Com}(\tilde{A},\tilde{B})$\,.
\end{theorem}
\begin{proof}
Let $A=U|A|$ and $B=V|B|$ be the polar decompositions of $A$ and
$B$, respectively. Let $\{p_n\}$ be a sequence of polynomials with
no constant term such that $p_n(t)\rightarrow {t^{1\over2}}$
uniformly on a certain compact set as $n\rightarrow\infty$. Let
$X\in\mbox{Com}(A,B)$. By our hypothesis, $A^* X=XB^*$. Hence
$|A|^2X=X|B|^2$ and so
 $p_n(|A|^2)X=Xp_n(|B|^2)$, hence $|A|X=X|B|$. Using the same argument we get
$|A|^{1\over2}X=X|B|^{1\over2}$. Thus $U|A|^nX=XV|B|^n$ for $n\in
\mathbb{N}$. We can use the argument above to show that
$Up_n(|A|)X=XVp_n(|B|)$ and conclude that
$U|A|^{1\over2}X=XV|B|^{1\over2}$. Hence we have
$$\tilde{A}X=|A|^{1\over2}U|A|^{1\over2}X=|A|^{1\over2}XV|B|^{1\over2}=X|B|^{1\over2}V|B|^{1\over2}=X\tilde{B}\,.$$
\end{proof}

\begin{corollary}
Let $A\in \mathbb{B}(\mathscr{H}_1)$ and $B^*\in
\mathbb{B}(\mathscr{H}_2)$ be either log-hyponormal or
$p$-hyponormal
 operators. Then $\mbox{Com}(A,B)\subseteq\mbox{Com}(\tilde{A},\tilde{B})$
\end{corollary}
\begin{proof}
It follows from Lemma \ref{L3} and Theorem \ref{C1}.
\end{proof}

Let $A=U|A|$ be the polar decomposition of $A$.
$A_{(s,t)}=|A|^sU|A|^t$, for $s,t\geq0$ is called $(s,t)-$Aluthge
transform of $A$. Note that we can use the proof of Theorem 3.1 for
$A_{(s,t)}, B_{(s,t)}$ instead of $\tilde{A}, \tilde{B}$,
respectively.

Using some ideas of \cite[Theorem 8]{JTU} we prove the next
result.
\begin{theorem}\label{C3}
Let $A\in \mathbb{B}(\mathscr{H}_1)$ be invertible and $B\in
\mathbb{B}(\mathscr{H}_2)$ be arbitrary. If $(A,B)$ has the
FP-property, then $\mbox{Com}(A,B)=\mbox{Com}(\tilde{A},\tilde{B})$.
\end{theorem}
\begin{proof} It is sufficient to prove that $\mbox{Com}(A,B)\supseteq\mbox{Com}(\tilde{A},\tilde{B})$.
Let $A=U|A|$ and $B=V|B|$ be the polar decompositions of $A$ and
$B$, respectively and $X\in{Com}(\tilde{A},\tilde{B})$. Let
$W=|A|^{-1\over 2}X|B|^{1\over2}$. Since $\tilde{A}X=X\tilde{B}$, we
have
$$|A|^{-1\over2}\tilde{A}X|B|^{1\over2}=|A|^{-1\over2}X\tilde{B}|B|^{1\over2}\,,$$
$$\:\:\:\:\:\:\:\:\:\:\:\:\:\:\:\:\:\:\:U|A|^{1\over2}X|B|^{1\over2}=|A|^{-1\over2}X|B|^{1\over2}V|B|^
{1\over2}|B|^{1\over2}\,,$$
$$U|A||A|^{-1\over2}X|B|^{1\over2}=|A|^{-1\over2}X|B|^{1\over2}V|B|\,,$$
$$AW=WB.$$
Hence by hypothesis and Lemma \ref{L1}, $\overline{R(W)}$ reduces
$A$, $N(W)^\bot$ reduces $B$ and $A|_{\overline{R(W)}}$ and
$B|_{N(W)^\bot}$ are normal operators. Therefore
$$A=N\oplus S\,\,\, \mbox{on\,}\,\overline{R(W)}\oplus R(W)^\bot $$
and
$$B=M\oplus T\,\,\mbox{on\,}\,N(W)^\bot\oplus N(W)\,,$$
where $N$ and $M$ are unitarily equivalent normal operators.
Operator $A$ is invertible and so are $N$ and $S$. Since $N$ and $M$
are unitarily equivalent, $M$ is invertible. Let
$$X=\left(
     \begin{array}{cc}
      X_1 & X_2 \\
      X_3 & X_4 \\
     \end{array}
    \right)
\,\,\mbox{and}\,\, W=\left(
    \begin{array}{cc}
     W_1 & 0 \\
     0 & 0 \\
    \end{array}
   \right)
$$
with respect to $\mathscr{H}_1=\overline{R(W)}\oplus R(W)^\bot$ and $\mathscr{H}_2=N(W)^\bot\oplus N(W)$. Clearly
$|A|^{-1}=|N|^{-1}\oplus|S|^{-1}$. It follows from $W=|A|^{-1\over 2}X|B|^{1\over2}$ that
$$\left(
  \begin{array}{cc}
   W_1 & 0 \\
   0 & 0 \\
  \end{array}
 \right)
=\left(
 \begin{array}{cc}
  |N|^{-1\over2}X_1|M|^{1\over2} & |N|^{-1\over2}X_2|T|^{1\over2} \\
  |S|^{-1\over2}X_3|M|^{1\over2} & |S|^{-1\over2}X_4|T|^{1\over2} \\
 \end{array}
\right).
$$
Hence $X_2|T|^{1\over2}=0$, $X_3=0$, $X_4|T|^{1\over2}=0$ so
$X_2\tilde{T}=0$ and $X_4\tilde{T}=0$. Then $\tilde{A}X=X\tilde{B}$
implies that
$$\left(
  \begin{array}{cc}
   NX_1 & NX_2 \\
   0 & \tilde{S}X_4 \\
  \end{array}
 \right)
=\left(
 \begin{array}{cc}
  X_1M & 0 \\
  0 & 0 \\
 \end{array}
\right).
$$
Hence $X_2=0$ and $X_4=0$. Since $\tilde{A}=N\oplus \tilde{S}$ and
$\tilde{B}=M\oplus\tilde{T}$ and $\tilde{A}X=X\tilde{B}$ and
$X=X_1\oplus0$, we have $NX_1=X_1M$ and this, in turn, implies that
$AX=XB$.
\end{proof}
\begin{remark}
In the preceding theorem if we assume that both $A$ and $B$ are
invertible, then we can easily prove the theorem. To see this, let
$Y\in\mbox{Com}(\tilde{A},\tilde{B})$. Then by Lemma \ref{A3}{\rm
(i)} $|A|^{-1\over2}Y|B|^{1\over2}\in\mbox{Com}(A,B).$ It follows
from the FP-property for $(A,B)$ and Remark \ref{A2}
$$Y=|A|^{1\over2}|A|^{-1\over2}Y|B|^{1\over2}|B|^{-1\over2}=|A|^{-1\over2}Y|B|^{1\over2}\in\mbox{Com}(A,B).$$
\end{remark}
\begin{corollary}
Let $A\in \mathbb{B}(\mathscr{H}_1)$ be log-hyponormal operator and
$B^*\in \mathbb{B}(\mathscr{H}_2)$ be either p-hyponormal or
log-hyponormal
 operator, then $\mbox{Com}(A,B)=\mbox{Com}(\tilde{A},\tilde{B})$.
\end{corollary}
\begin{proof}
It follows from Lemma \ref{L3} and Theorem \ref{C3}.
\end{proof}

\begin{corollary}
Let $A\in \mathbb{B}(\mathscr{H}_1)$ and $B\in
\mathbb{B}(\mathscr{H}_2)$ be invertible operators. If $(A,B)$ has
the FP-property and $n\in\mathbb{N}$, then
$\mbox{Com}(\Delta_n(A),\Delta_n(B))=\mbox{Com}(A,B)$.
\end{corollary}
\begin{proof}
Let $\Delta_n(A)Y=Y\Delta_n(B)$ for some
$Y\in\mathbb{B}(\mathscr{H}_2,\mathscr{H}_1)$. By Corollary
\ref{A6}, we see that $(\Delta_{n-1}(A),\Delta_{n-1}(B))$ has the
FP-property, so by Theorem \ref{C3} we have
$\Delta_{n-1}(A)Y=Y\Delta_{n-1}(B)$. Repeating this process we
conclude the result as desired. For the revers similar argument can
be applied.
\end{proof}

\section{Fuglede--Putnam--Aluthge problem modulo trace ideals}

In this section, we present some results about the
Fuglede--Putnam--Aluthge problem modulo trace ideals. We obtain some
inequalities related to this problem by using some ideas of
\cite{KIT1}.

\begin{lemma}\label{K1}
Let $A=U|A|$ be the polar decomposition of $A$ and $X\in
\mathbb{B}(\mathscr{H})$ be a self-adjoint operator such that
$\mbox{Re}( U|A|^{1\over2})\geq a>0$ and $U^* X=XU$. Then
$$\|\tilde{A}^* X-X\tilde{A}\|_p\geq 2a\|\,|A|^{1\over2}X-X|A|^{1\over2}\|_p$$
for $1\leq p\leq\infty.$
\end{lemma}
\begin{proof}
We consider two
cases:

Case (i). $p=\infty$. Clearly
$(|A|^{1\over2}X-X|A|^{1\over2})^*=-(|A|^{1\over2}X-X|A|^{1\over2})$.
It follows from \cite[ Theorem 2.4.1.16]{FUR} that there exist a
sequence $\{f_n\}_{n\in \mathbb{N}}$ of unit vectors in
$\mathscr{H}$ and number $t\in {\rm
sp}(|A|^{1\over2}X-X|A|^{1\over2})$ such that $\overline{t}=-t$,
$(|A|^{1\over2}X-X|A|^{1\over2}-t)f_n\rightarrow0$ as
$n\rightarrow\infty$ and
$|t|=\|\,|A|^{1\over2}X-X|A|^{1\over2}\|_{\infty}$. Now
\begin{eqnarray*}
&&\hspace{-2.5cm}\|\tilde{A}^* X-X\tilde{A}\|_{\infty}\\
&\geq& |\langle\tilde{A}^*X-X\tilde{A}f_n,f_n\rangle|\\
&=&|\langle|A|^{1\over2}U^*|A|^{1\over2}X-X|A|^{1\over2}U|A|^{1\over2}f_n,f_n\rangle|\\
&=&|\langle(|A|^{1\over2}U^*(|A|^{1\over2}X-X|A|^{1\over2})+(|A|^{1\over2}X-X
|A|^{1\over2})U|A|^{1\over2}\\
&&+|A|^{1\over2}(U^*
X-XU)|A|^{1\over2})f_n,f_n\rangle|\\
&=&|\langle|A|^{1\over2}U^*(|A|^{1\over2}X-X|A|^{1\over2}-t)f_n,f_n\rangle\\
&&+
\langle(|A|^{1\over2}X-X|A|^{1\over2}-t)U|A|^{1\over2}f_n,f_n\rangle+t\langle|A|^{1\over2}U^*+
U|A|^{1\over2}f_n,f_n\rangle|\\
&\geq& |t|\langle|A|^{1\over2}U^*+U|A|^{1\over2}f_n,f_n\rangle-|\langle|A|^{1\over2}U^*
(|A|^{1\over2}X-X|A|^{1\over2}-t)f_n,f_n\rangle\\
&&+\langle(|A|^{1\over2}X-X|A|^{1\over2}-t)U|A|^{1\over2}f_n,f_n\rangle|.
\end{eqnarray*}
We observe that
$$|\langle|A|^{1\over2}U^*(|A|^{1\over2}X-X|A|^{1\over2}-t)f_n,f_n\rangle+
\langle(|A|^{1\over2}X-X|A|^{1\over2}-t)U|A|^{1\over2}f_n,f_n\rangle|\to0$$
as $n\rightarrow\infty$. Hence $$\|\tilde{A}^*
X-X\tilde{A}\|_{\infty}\geq
2a\|\,|A|^{1\over2}X-X|A|^{1\over2}\|_{\infty}.$$

 Case (ii). $1\leq p<\infty$.
We can assume that $\tilde{A}^* X-X\tilde{A}\in \mathcal{C}_p$ and
hence it is compact. If
$\pi:\mathbb{B}(\mathscr{H})\rightarrow{\mathbb{B}(\mathscr{H})\over{\mathcal{C}_\infty}}$
is the quotient map then we have $\pi(\tilde{A}^* X-X\tilde{A})=0$.
It is obvious that $\pi(A)=\pi(U)\pi(|A|)$ is the polar
decomposition of $\pi(A)$. Since $U^* X=XU$ we have $\pi(U^*
)\pi(X)=\pi(X)\pi(U)$. Hence $\pi(|A|^{1\over2}X-X|A|^{1\over2})=0$
by Case (i). So $|A|^{1\over2}X-X|A|^{1\over2}$ is a compact normal
operator. It is therefore diagonalizable and hence there exist an
orthonormal basis $\{e_n\}_{n\in\mathbb{N}}$ of $\mathscr{H}$ and
numbers $t_n$ such that $(|A|^{1\over2}X-X|A|^{1\over2})e_n=t_ne_n$.
Thus the $|t_n|$'s are the singular values of
$|A|^{1\over2}X-X|A|^{1\over2}$ and
\begin{eqnarray*}
\|\tilde{A}^* X-X\tilde{A}\|_p^p&\geq&\sum_{n=1}^\infty|\langle\tilde{A}^*
X-X\tilde{A}e_n,e_n\rangle|^p\\
&=&\sum_{n=1}^\infty|\langle|A|^{1\over2}U^*|A|^{1\over2}X-X|A|^{1\over2}U|A|^{1\over2}e_n,e_n\rangle|^p\\
&=&\sum_{n=1}^\infty|\langle(|A|^{1\over2}U^*(|A|^{1\over2}X-X|A|^{1\over2})+(|A|^{1\over2}X-X
|A|^{1\over2})U|A|^{1\over2}\\
&&+|A|^{1\over2}(U^*X-XU)|A|^{1\over2})e_n,e_n\rangle|^p\\
&=&\sum_{n=1}^\infty|t_n|^p|(\langle|A|^{1\over2}U^*+U|A|^{1\over2}e_n,e_n\rangle)|^p\\
&\geq&(\sum_{n=1}^\infty|t_n|^p)(2a)^p=(2a)^p\|\,|A|^{1\over2}X-X|A|^{1\over2}\|_p^p\,.
\end{eqnarray*}
Thus
\begin{eqnarray*}\label{is2}
\|\tilde{A}^* X-X\tilde{A}\|_p\geq
2a\|\,|A|^{1\over2}X-X|A|^{1\over2}\|_p\,.
\end{eqnarray*}

\end{proof}
Now we get our last main result.
\begin{theorem}\
Let $A=U|A|$ and $B=V|B|$ be the polar decompositions of $A$ and
$B$, respectively, and $X\in \mathbb{B}(\mathscr{H})$ such that
$\mbox{Re}(U|A|^{1\over2})\geq a>0$ and
$\mbox{Re}(V|B|^{1\over2})\geq a>0$ and $U^* X=XV$. Then
$$\|\tilde{A}^*
 X-X\tilde{B}\|_p\geq2a\|\,|A|^{1\over2}X-X|B|^{1\over2}\|_p$$
for $1\leq p\leq \infty$.
\end{theorem}
\begin{proof}
Let $T= \left(
 \begin{array}{cc}
  A & 0 \\
  0 & B \\
 \end{array}
\right)$ and $Y= \left(
 \begin{array}{cc}
  0 & X \\
  X^* & 0 \\
 \end{array}
\right)$. Then $Y$ is self-adjoint. Let
$T=W|T|$ be the polar decomposition of $T$. Note that
$W=\left(
 \begin{array}{cc}
  U & 0 \\
  0 & V \\
 \end{array}
 \right)$ and hence $W^* Y=YW$ by the assumption $U^* X=XV$.

\noindent Also $W|T|^{1\over 2}=\left(
 \begin{array}{cc}
  U|A|^{1\over2} & 0 \\
  0 & V|B|^{1\over2} \\
 \end{array}
 \right)\geq a\geq0$ so we have $\|\tilde{T}^* Y-Y\tilde{T}\|_p\geq 2a\|\,|T|^{1\over2}Y-Y|T|^{1\over2}\|_p$
 by Lemma \ref{K1}. Since
 $\tilde{T}=\left(\begin{array}{cc}
  \tilde{A} & 0 \\
  0 & \tilde{B} \\
 \end{array}
 \right)$ and $|T|^{1\over2}=\left(\begin{array}{cc}
  |A|^{1\over2} & 0 \\
  0 & |B|^{1\over2} \\
 \end{array}
 \right)$, a simple computation shows that
\begin{eqnarray*}
&\left\|\left(\begin{array}{cc}
  0 & \tilde{A}^* X-X\tilde{B} \\
  \tilde{B}^* X^*-X^*\tilde{A} & 0 \\
 \end{array}
 \right)\right\|_p^p\\
 &\geq2^pa^p\left\|\left(\begin{array}{cc}
  0 & |A|^{1\over2}X-X|B|^{1\over2} \\
  |B|^{1\over2}X^*-X^*|A|^{1\over2} & 0 \\
 \end{array}
 \right)\right\|_p.
\end{eqnarray*}
Utilizing \eqref{1new} we obtain $$\|\tilde{A}^*
 X-X\tilde{B}\|_p\geq2a\|\,|A|^{1\over2}X-X|B|^{1\over2}\|_p\,.$$
\end{proof}

\begin{corollary}
Let $A=U|A|$ and $B=V|B|$ be the polar decompositions of $A$ and
$B$, respectively, and $X\in \mathbb{B}(\mathscr{H})$ such that
$\mbox{Re}(U|A|^{1\over2})\geq a>0$ and
$\mbox{Re}(V|B|^{1\over2})\geq a>0$ and $U^* X=XV$ and $\tilde{A}^*
X-X\tilde{B}\in \mathcal{C}_p $ for some $1\leq p\leq\infty$. Then
$|A|^{1\over2}X-X|B|^{1\over2}\in \mathcal{C}_p.$
\end{corollary}

\begin{corollary}\label{K2}
Let $A=U|A|$ and $B=V|B|$ be the polar decompositions of $A$ and
$B$, respectively, and $X\in \mathbb{B}(\mathscr{H})$ such that
$\mbox{Re}(U|A|^{1\over2})\geq a>0$ and
$\mbox{Re}(V|B|^{1\over2})\geq a>0$ and $U^* X=XV$ and $\tilde{A}^*
X=X\tilde{B}$, then $|A|X=X|B|$.
\end{corollary}

\begin{remark}Under the conditions of Corollary \ref{K2} we have
$$A^*|A|X=|A|U^* |A|X=|A|U^* X|B|=|A|XV|B|=|A|XB.$$
Hence there exists an operator $Y(=|A|X)$ such that $A^* Y=YB$.
\end{remark}

\begin{remark}
For $\delta>0$, let $\mbox{Com}_\delta(A,B)$ be the set of
all operators $X\in\mathbb{B}(\mathscr{H}_2,\mathscr{H}_1)$ such
that $\|AX-XB\|\leq\delta$. Moore \cite{MOO} proved that for any $\varepsilon >0$ there exists $\delta >0$ such that
$$\mbox{Com}_\delta(A,B)\cap \mathbb{B}(\mathscr{H})_1 \subseteq \mbox{Com}_\varepsilon(A^*,B^*)\,,$$
where $ \mathbb{B}(\mathscr{H})_1$ denotes the closed norm-unit ball of $ \mathbb{B}(\mathscr{H})$.
Let $A$ be an operator with the polar decomposition $U|A|$ and
$X\in\mbox{Com}_\delta(|A|^{1\over2},|A|^{1\over2})\bigcap\mbox{Com}_\delta(U^\ast,U)$
for some $\delta>0$. From
\begin{eqnarray*}\|\tilde{A}^\ast
X-X\tilde{A}\|&=&\||A|^{1\over2}U^\ast(|A|^{1\over2}X-X|A|^{1\over2})\\&+&(|A|^{1\over2}X-X|A|^{1\over2})U|A|^{1\over2}+|A|^{1\over2}(U^\ast
X-XU)|A|^{1\over2}\|\\&\leq&(\||A|^{1\over2}U^\ast\|+\|U|A|^{1\over2}\|+\|A\|)\delta\\&=&(2\|A\|^{1\over2}+\|A\|)\delta
\end{eqnarray*}
we can see that
$X\in\mbox{Com}_{\varphi_A(\delta)}(\tilde{A}^\ast,\tilde{A})$ for some positive increasing function $\varphi_A(t)$ on $(0,\infty)$. Thus
$$\mbox{Com}_\delta(|A|^{1\over2},|A|^{1\over2})\bigcap\mbox{Com}_\delta(U^\ast,U)\subseteq\mbox{Com}_{\varphi_A(\delta)}(\tilde{A}^\ast,\tilde{A}).$$
Now by Lemma \ref{K1} if $\mbox{Re}(U|A|^{1/2})\geq a>0$ it is easy
to see that
$$\mbox{Com}_{\varphi_A(\delta)}(\tilde{A}^\ast,\tilde{A})\bigcap\mbox{Com}(U^\ast,U)\subseteq\mbox{Com}_{\psi(\delta)}(|A|^{1\over2},|A|^{1\over2})\,,$$
where $\psi(t)=\varphi_A(t)/(2a)$.
\end{remark}

\textbf{Acknowledgement.} The authors would like to sincerely thank
Professor T. Ando for very useful comments improving the paper.

\end{document}